\documentclass[journal,twoside,web]{ieeecolor}
\usepackage{lcsys}
\usepackage{graphicx}
\usepackage{textcomp}

\newif\ifarxiv

\arxivtrue

\usepackage[utf8]{inputenc} 
\usepackage[T1]{fontenc}    
\usepackage[hidelinks]{hyperref}       
\usepackage{url}            
\usepackage{booktabs}       
\usepackage{amsfonts}       
\usepackage{nicefrac}       
\usepackage{microtype}      
\usepackage[usenames, dvipsnames]{xcolor}         
\usepackage{amsmath,amsthm,amssymb}

\usepackage{wrapfig}

\usepackage{amssymb}

\usepackage{cite}

\usepackage{amsthm}

\usepackage{glossaries}
\usepackage{url}
\usepackage{siunitx}
\usepackage{tcolorbox} 

\usepackage{algorithm}
\usepackage{algpseudocode}
\algnewcommand\algorithmicleftcomment[1]{\(\triangleright\) \textit{#1}}%

\algrenewcommand\algorithmiccomment[1]{\hfill\(\triangleright\) \textit{#1}}%
\algdef{S}[WHILE]{WhileNoDo}[1]{\algorithmicwhile\ #1}%
   
\usepackage{multicol}
\usepackage{multirow}
\usepackage{pifont}

\usepackage[font={small}]{caption}
\usepackage{subcaption}

\DeclareMathOperator*{\argmax}{arg\,max}
\DeclareMathOperator*{\argmin}{arg\,min}

\newcommand{\edt}[1]{{\color{black} #1}}

\newcommand{\safe}{\mathrm{s}}

\newcommand{\expec}[1]{\mathbb{E}\left[ #1 \right]}

\newcommand{\pv}{\mathbf{v}}
\newcommand{\px}{\mathbf{x}}
\newcommand{\traj}{\vec{\mathbf{x}}}

\newcommand{\pz}{\mathbf{z}}

\newcommand{\sasvec}{\mathbf{h}(\px)}
\newcommand{\bvec}{\mathbf{b}}
\newcommand{\betavec}{\mathbf{\beta}}

\newcommand{\reals}{\mathbb{R}}
\newcommand{\polys}{\mathbb{P}[\px]}

\newcommand{\sospolys}{\Sigma^2[\px]}

\newcommand{\up}[1]{\overline{#1}}
\newcommand{\low}[1]{\underline{#1}}

\theoremstyle{definition}
\newtheorem{definition}{Definition}
\newtheorem{problem}{Problem}

\theoremstyle{plain}
\newtheorem{theorem}{Theorem}
\newtheorem{corollary}[theorem]{Corollary}
\newtheorem{proposition}[theorem]{Proposition}
\newtheorem{lemma}[theorem]{Lemma}

\theoremstyle{plain}
\newtheorem{remark}{Remark}

\usepackage{fancyhdr}

\begin{document}

\title{
On Polynomial Stochastic Barrier Functions: \\Bernstein Versus Sum-of-Squares
}

\author{Peter Amorese and Morteza Lahijanian
\thanks{
Submitted for review on March 17 2025. 
This work was supported by the National Science Foundation (NSF) Center for Autonomous Air Mobility \& Sensing (CAAMS) under award 2137269.
}
\thanks{Authors are with the Dept. of Aerospace Eng. Sciences at University of Colorado Boulder, USA
       \texttt{\{firstname.lastname\}@colorado.edu}}%
}

\maketitle

\ifarxiv
    \fancypagestyle{firstpage}{
        \fancyhf{} 
        \fancyhead[C]{To appear in IEEE Control Systems Letters (L-CSS) 2025} 
        \renewcommand{\headrulewidth}{0.0pt} 
    }
    \thispagestyle{firstpage}
\else
    \thispagestyle{empty}
\fi

\begin{abstract}

    Stochastic Barrier Functions (SBFs) certify the safety of stochastic systems by formulating a functional optimization problem, which state-of-the-art methods solve using Sum-of-Squares (SoS) polynomials.
    This work focuses on polynomial SBFs and introduces a new formulation based on Bernstein polynomials and provides a comparative analysis of its theoretical and empirical performance against SoS methods. We show that the Bernstein formulation leads to a \textit{linear program} (LP), in contrast to the \emph{semi-definite program} (SDP) required for SoS, and that its relaxations exhibit favorable theoretical convergence properties. However, our empirical results reveal that the Bernstein approach struggles to match SoS in practical performance, exposing an intriguing gap between theoretical advantages and real-world feasibility.
\end{abstract}


\section{Introduction}


\IEEEPARstart{R}{ealistic} systems often exhibit complex dynamics and are subject to uncertainty, making \emph{safety assurance} a significant challenge in \emph{safety-critical} applications. To accurately capture their behavior, it is essential to consider both nonlinear and stochastic dynamical models. However, formal safety verification for such models remains highly challenging. \emph{Stochastic Barrier Functions} (SBFs) \cite{prajna2007framework, santoyo2021barrier}
provide \textit{safety certificates}, i.e., quantitative bounds on the probability that a system remains safe over a given time horizon. 
\edt{Synthesizing a SBF is a functional optimization problem, relying on function templates, with Sum-of-Squares (SoS) polynomials}
\cite{santoyo2021barrier}
being a widely used and effective approach. While SoS-based methods have demonstrated success, alternative polynomial representations, such as Bernstein polynomials, remain largely unexplored in this context. This work investigates the use of Bernstein polynomials for SBF synthesis and compares their performance against the established SoS-based approach, assessing their advantages and disadvantages for safety verification in stochastic dynamical systems.

A key advantage of SBFs for safety guarantees is that they bypass the challenge of uncertainty propagation through nonlinear dynamics \cite{prajna2007framework}, instead focusing on finding an \emph{energy function} \edt{of the state} that satisfies certain constraints over regions of interest. \edt{Intuitively, if initial states have low energy and unsafe regions have high energy, then bounding the system's increase in energy in a single step anywhere bounds the probability of being unsafe.} If these constraints hold, the SBF provides a guaranteed lower bound on the system's probability of safety, serving as a \emph{barrier certificate}. The choice of the candidate function affects the tightness of these bounds, motivating SBF synthesis as a functional optimization problem. \edt{To achieve informative certificates for complex systems}, the candidate function must be expressed using a basis that allows i) universal function approximation and ii) efficient bounding over a given state set, giving rise to many choices of function templates \cite{mathiesen2022safety, mazouz2024piecewise, jagtap2020formal}. However, bounding \edt{an arbitrary} nonlinear, potentially non-convex function over a set is NP-hard~\cite{murty1985some}, even for polynomials.


SoS relaxation techniques address this by providing a systematic way to determine polynomial non-negativity, enabling function bounding over a given domain via Semi-Definite Program (SDP) optimization \cite{stengle1974nullstellensatz}. 
Alternatively, Bernstein polynomials provide efficient bounds using simple coefficient transformations \cite{garloff1985convergent}. Prior work~\cite{ben2016linear} compares SoS and Bernstein methods for Lyapunov function synthesis, showing that Bernstein polynomials offer better numerical stability with comparable computational effort. However, Bernstein polynomials have not yet been explored for SBF synthesis.


This paper presents the \emph{first} formulation of SBF synthesis using Bernstein polynomials, to the best of our knowledge, and provides a comprehensive analysis comparing Bernstein and SoS methods. Our key contributions are threefold. 
(i) We show that SBF synthesis can be formulated as a \emph{linear program} (LP) using Bernstein polynomial relaxations, offering a computationally efficient alternative to SoS-based approaches (SDP optimization). 
(ii) We analyze and compare the theoretical convergence rates of different techniques for obtaining tighter probability bounds for both Bernstein and SoS methods. 
(iii) Through empirical evaluations, we assess the performance of SoS and Bernstein methods, juxtaposing the enticing theoretical advantages of Bernstein against the practical shortcomings compared to SoS. 



\paragraph*{Related Work} 

\edt{Abstraction-based methods~\cite{skovbekk2023formal, dutreix2022abstraction, cauchi2019efficiency} offer an alternative approach to verifying safety using dynamic programming on discretized Markov models.}
The connection between abstraction-based methods and SBFs has been explored in~\cite{laurenti2023unifying}. 
\edt{Neural network SBFs~\cite{mathiesen2022safety, vzikelic2023learning} offer strong function approximation capability, at the cost of requiring expensive bound-propagation-based constraint verification.}
\edt{Among SoS methods, control barrier verification has been formulated as a LP using Diagonally-Domaninant SoS \cite{pond2023fast} polynomials, which cannot universally approximate non-negative functions \cite{josz2017counterexample}, rendering the approach asymptotically incomplete.}

This work differs from the above approaches in both focus and scope. Specifically, we focus on polynomial SBFs and introduce a novel LP formulation for their synthesis. Additionally, we provide a direct comparison between Bernstein and SoS relaxation approaches, highlighting their respective advantages and limitations.

\section{Problem Formulation and Approach}

We consider a (polynomial) discrete-time stochastic system described by the stochastic difference equation
\begin{equation} 
    \label{eq:dynamics}
    \px_{k+1} = f(\px_{k}) + \pv_k,
\end{equation}
where $k \in \mathbb{N}_0$ is the time step, $\px_k \in \mathbb{R}^D$ is the state, $f:\mathbb{R}^D \to \mathbb{R}^D$ is the vector field, 
and $\pv_k \in \mathbb{R}^D$ is an i.i.d. random variable distributed according to $p(\pv)$ with finite moments, i.e., $\pv_k \sim p(\pv)$. 
Let $f_i(\px)$ denote the $i$-th dimension of $f(\px)$, i.e., $f(\px) = [f_1(\px), \ldots, f_D(\px)]^\top$, and denote by $\polys$ the set of all $D$-variate polynomials on $\reals^D$.
We assume each $f_i(\px)$ is a polynomial, i.e., 
$f_i(\px) \in \polys$ for all $i\in\{1,\ldots,D\}$.

Let $X \subset \mathbb{R}^D$ be a compact set of interest, $X_s \subset X$ the safe set, and $X_0 \subseteq X_s$ the initial set.  We define $X_u = X \setminus X_s$ to be the unsafe set. 

\begin{definition} [Semi-Algebraic Set]
    A set $S \subset \reals^D$ is a \emph{semi-algebraic set} iff $S = \{\px \mid \sasvec \geq 0\}$ for some vector of polynomials $\sasvec = (h_1(\px), \dots h_r(\px)) \in (\polys)^r$. 
\end{definition}
\noindent We assume $X$, $X_s$, $X_u$, and $X_0$ are semi-algebraic sets. 




Given an initial state $\px_0 \in X_0$ and a finite horizon of interest $K \in \mathbb{N}_0$, let $\traj = \px_0 \px_1 \cdots \px_k$ denote a trajectory produced by System~\eqref{eq:dynamics}.
Since each state except $\px_0$ in the trajectory is a random variable, a probability density over trajectories is simply the joint distribution $p(\traj) = p(\px_1, \ldots, \px_k | \px_0)$.
A trajectory is safe iff for all $k \in \{0, \ldots K\}$, $\px_k \in X_s$ denoted $\traj \in X_s$. Therefore, the probability of safety over all $K$-step trajectories of System~\eqref{eq:dynamics} from a given $\px_0 \in X_0$ is
    \edt{$P_{\safe}(\traj,X_\safe, \px_0) = \int \mathbf{1}(\traj \in X_s ) p(\traj) d\px_1 \cdots d\px_K,$}
where $\mathbf{1}(\varphi) = 1$ iff predicate $\varphi$ is \emph{true}, otherwise $\mathbf{1}(\varphi) = 0$.

Finally, since $X_0$ is not necessarily a singleton, measures of safety probability must hold for all $\px_0 \in X_0$. 
Hence, we define the safety probability from initial set $X_0$ as:
    $P_{\safe}(\vec{\px},X_\safe, X_0) = \inf_{\px_0 \in X_0} P_{\safe}(\vec{\px}, X_\safe, \px_0)$.


Our goal is to find a lower bound on $P_\safe$ to guarantee the safety of System~\eqref{eq:dynamics}.
To this end, we focus on stochastic barrier functions (SBFs). 

\begin{definition}[SBF] \label{def:sbf}
    A \emph{stochastic barrier function} (SBF) for System~\eqref{eq:dynamics} is a scalar-valued function $B : X \to \mathbb{R}_{\geq 0}$ such that, for scalars $\eta, \gamma \in [0,1]$,
    \begin{subequations}
        \label{eq: barrier constraints}
        \begin{align}
            &B(\px) \geq 0 &&\forall \px \in X, \label{eq:constraint1} \\
            &B(\px) \geq 1 &&\forall \px \in X_u, \label{eq:constraint2} \\
            &B(\px) \leq \eta &&\forall \px \in X_0, \label{eq:constraint3} \\
            &\expec{B(f(\px) + \pv)} - B(\px) \leq \gamma &&\forall \px \in X_s. \label{eq:constraint4}
        \end{align}
    \end{subequations}

\end{definition}


%
\begin{theorem}[\edt{SBF} Certificate \cite{santoyo2021barrier}] \label{thm:sbfc}
    If there exists \edt{a} SBF $B(\px)$ for System~\eqref{eq:dynamics} 
    that satisfies Constraints \eqref{eq:constraint1}-\eqref{eq:constraint4} 
    with scalars $\eta$ and $\gamma$, then, for horizon $K \in \mathbb{N}$,
    $P_{\safe}(\vec{\px},X_\safe, X_0) \geq \delta_\safe$, where
    \begin{equation} \label{eq:psafe}
        \delta_{\safe} = 1 - (\eta + K \gamma).
        \vspace{-1mm}
    \end{equation}
\end{theorem}

Theorem~\ref{thm:sbfc} describes the safety certificate admitted by a given SBF. 
In this work, we consider the problem of \textit{synthesizing} a SBF, as stated below.

\begin{problem}[SBF Synthesis]
    \label{prob}
    Given stochastic System~\eqref{eq:dynamics} with a safe set $X_s \subset X$, initial set $X_0 \subseteq X_\safe$, time horizon $K \in \mathbb{N}_0$, and safety threshold $p_{\safe}$, find a SBF $B(\px)$ 
    that guarantees $P_\safe(\vec{x}, X_\safe,X_0) \geq p_\safe$.
\end{problem}

\textbf{Approach Overview. }
Synthesis of a SBF (Problem~\ref{prob}) is a functional optimization problem, where $B(\px)$, $\eta$, and $\gamma$ are decision variables with the objective of maximizing \eqref{eq:psafe}. Obtaining a witness of $B(\px)$ with $\eta$ and $\gamma$ such that \edt{$\delta_\safe >= p_\safe$} solves Problem~\ref{prob}.
\edt{However, generating an informative certificate requires the template for $B(\px)$ to approximate aribitrary functions well and admit tractable regional bounds required by constraints \eqref{eq:constraint1}-\eqref{eq:constraint4}.}


In this work, we restrict $B(\px)$ to the class of multivariate polynomials. To be precise, we distinguish between polynomials defined by their \textit{maximal} degree (seen in Bernstein formulations), i.e.,
\vspace{-3mm}
\begin{equation} \label{eq:B bernstein}
    B(\px) = \sum_{l_1 = 0}^m \cdots \sum_{l_D = 0}^m b_{l_1 \ldots l_D} \prod_{j=1}^D x_j^{l_j},
\end{equation}
with coefficients $b_{l_1 \ldots l_D} \in \reals$,
and those defined by their \textit{cumulative} degree (seen in SoS formulations), i.e.,
\vspace{-1mm}
\begin{equation} \label{eq:B SoS}
    B(\px) = \sum_{|\vec{l}| \leq m} b_{l_1 \ldots l_D} \prod_{j=1}^D x_j^{l_j},
    \vspace{-1mm}
\end{equation} 
where $\vec{l} = [l_1, \ldots, l_D]$ and $|\vec{l}| = \sum_d l_d$. The multi-sum, e.g., in ~\eqref{eq:B bernstein}, is denoted $\sum_{\vec{l} = \vec{0}}^{\vec{m}}$ where $\vec{m} = [m, ..., m]$.

In the following sections, we discuss two polynomial relaxation methods that allow us to bound the extrema of $B(\px)$: Sum-of-Squares (SoS) programming, and our alternative approach using Bernstein polynomials. We also provide a detailed comparison analysis of these methods.

\textbf{Notations. } 
For simplicity, let $\bvec = [b_{0\ldots 0}, \ldots, b_{m \ldots m}]^\top$ denote an organized vectorization of the coefficients of a polynomial with coefficients $b_{l_1 \ldots l_D}$. Let $\mathbf{z}(\px) = [1, x_1, \ldots, x_1^m \cdots x_D^m]^\top$ be the vector containing all unique monomials such that any polynomial can be written as $a(\px) = \bvec^\top \mathbf{z}(\px) $. The polynomial degree is denoted $deg(a(\px))$.



\begin{figure*}[h!]
    \centering

    \begin{subfigure}{0.48\textwidth}
        \centering
        \caption*{\footnotesize $\lambda(\px) = (x - 0.5)^2$}
        \begin{subfigure}[t]{0.48\textwidth}
            \centering
            \includegraphics[width=\linewidth]{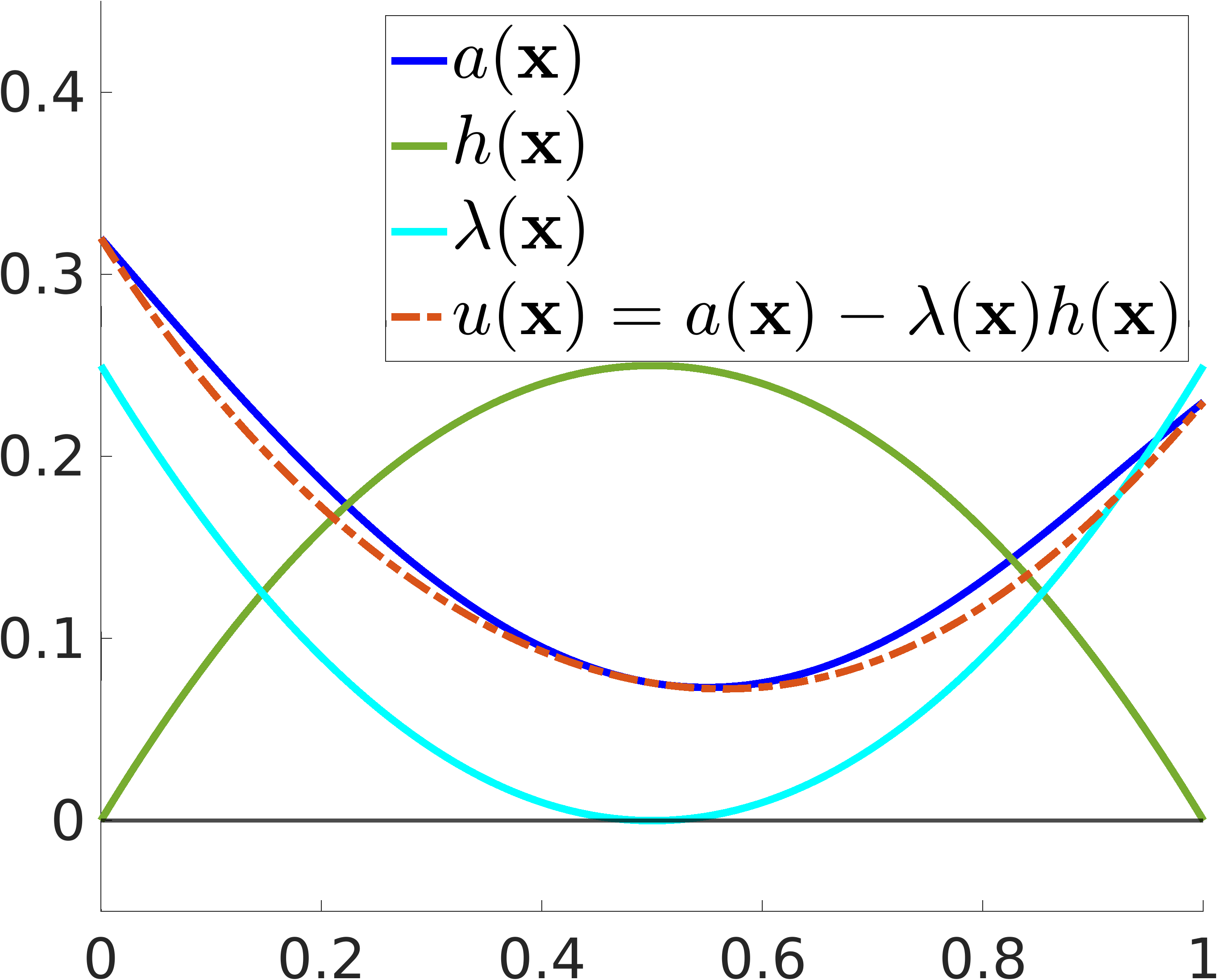}
            \caption{$u(\px)$ over $S$ ($deg(\lambda) = 2$)}
            \label{fig:lagr_ex_d2_set}
        \end{subfigure}
        \hfill
        \begin{subfigure}[t]{0.48\textwidth}
            \centering
            \includegraphics[width=\linewidth]{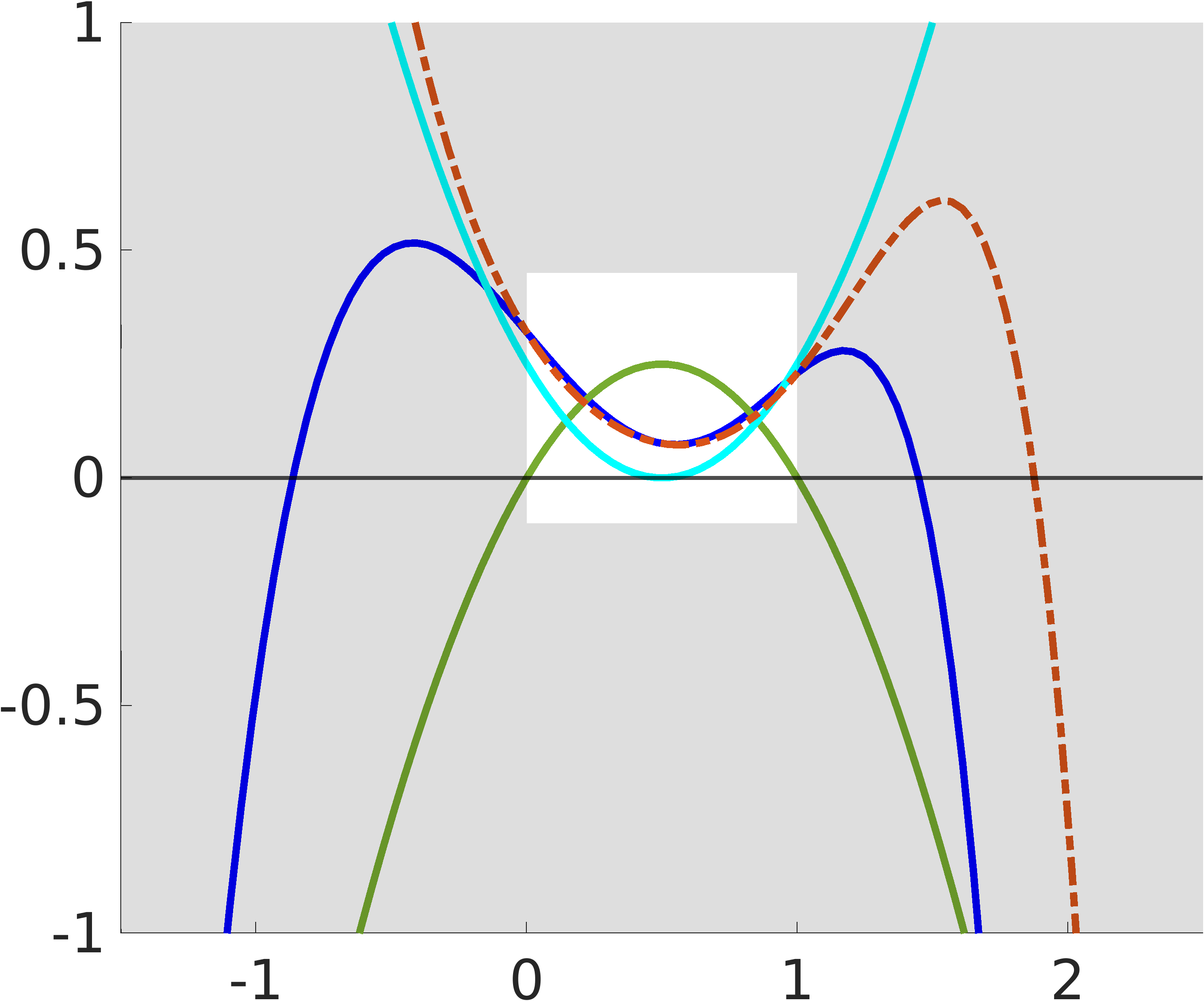}
            \caption{$u(\px)$ over $\reals$ ($deg(\lambda) = 2$)}
            \label{fig:lagr_ex_d2_global}
        \end{subfigure}
        \label{fig:lagr_ex_d2}
    \end{subfigure}
    \hfill
    \begin{subfigure}{0.48\textwidth}
        \centering
        \caption*{\footnotesize $\lambda(\px) = (x - 0.5)^6$}
        \begin{subfigure}[t]{0.48\textwidth}
            \centering
            \includegraphics[width=\linewidth]{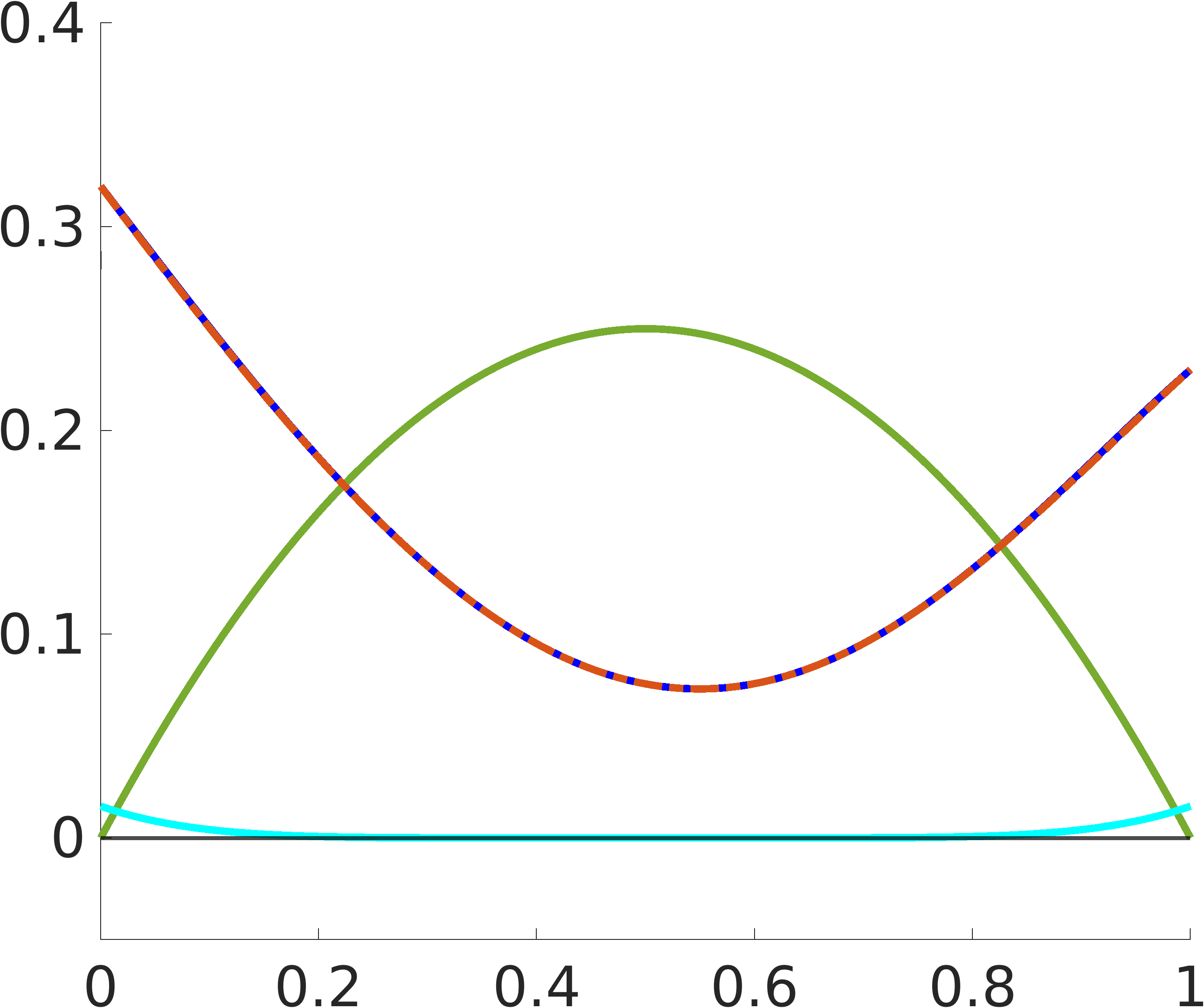}
            \caption{$u(\px)$ over $S$ ($deg(\lambda) = 6$)}
            \label{fig:lagr_ex_d6_set}
        \end{subfigure}
        \hfill
        \begin{subfigure}[t]{0.48\textwidth}
            \centering
            \includegraphics[width=\linewidth]{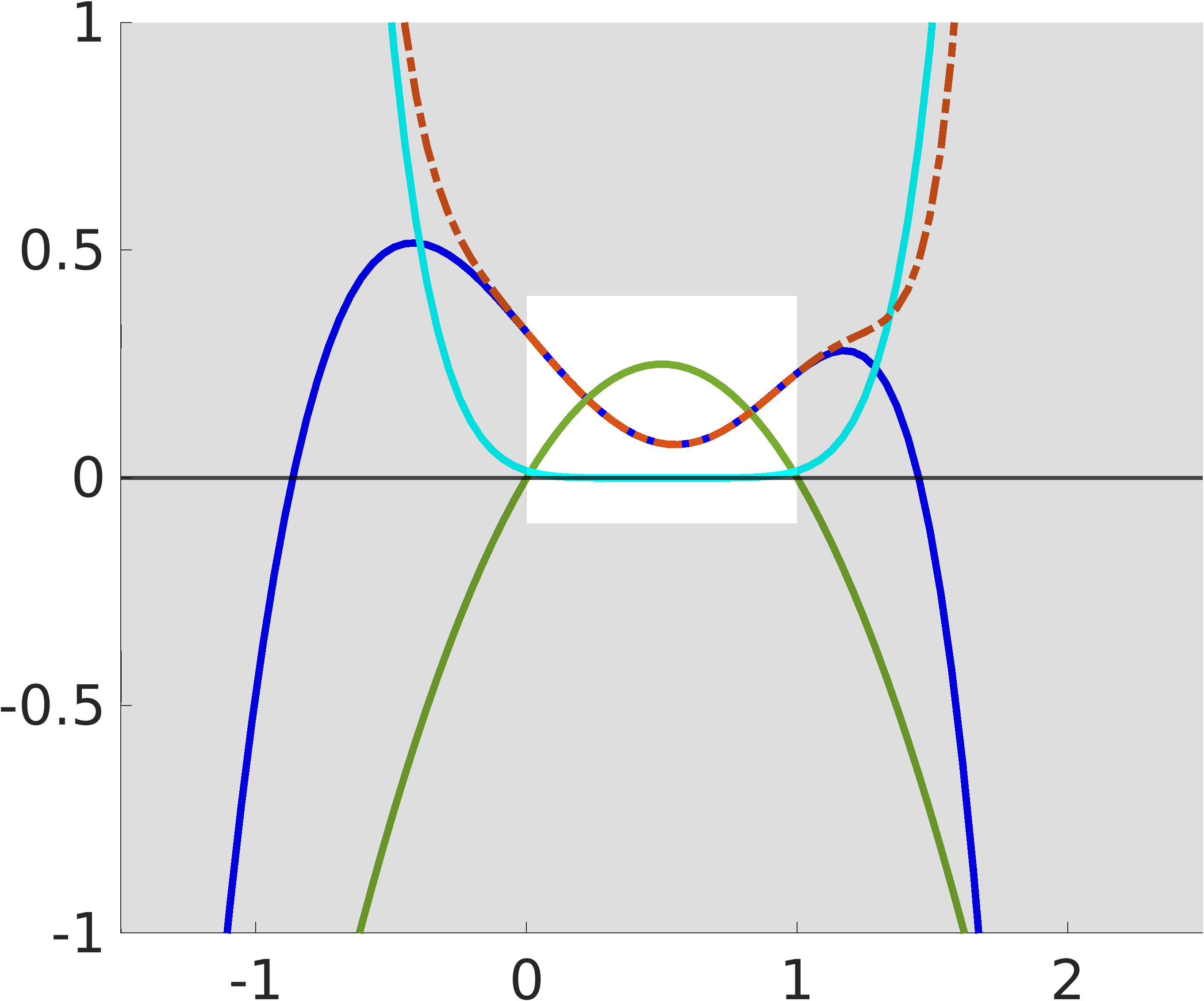}
            \caption{$u(\px)$ over $\reals$ ($deg(\lambda) = 6$)}
            \label{fig:lagr_ex_d6_global}
        \end{subfigure}
        \label{fig:lagr_ex_d6}
    \end{subfigure}

    \caption{\textit{Lagrangian Multiplier Example}: Using Theorem \ref{thm:putinar}, we aim to assert the non-negativity of a degree 6 polynomial $p(\px)$ over $S = [0,1]$ described by $h(\px) = x^2 - x$. Figure (a, b) show results for a degree 2 Lagrange multiplier; (c, d) show results using degree 6. }
    \label{fig:lagr_ex}
    \vspace{-4mm}
\end{figure*}

\section{Preliminaries: SoS SBFs}
\label{sec:prelim}
Here, we provide a brief review of the Sum-of-Squares (SoS) formulation for SBFs shown in \cite{santoyo2021barrier}. Let $\sospolys$ define the set of all \textit{Sum-of-Squares polynomials} such that
    $\sospolys = \big\{ q(\px)\in \polys \mid q(\px) = \sum_{i=1}^N a_i^2(\px), a_i(\px) \in \polys, \; N\in \mathbb{N}  \big\}$.
It is straightforward to observe that all SoS polynomials are non-negative over $\reals^D$. 
One can equivalently express $\sospolys$ as the set of all $q(\px) = \pz^\top\!(\px) Q \pz(\px)$ such that $Q \in \reals^{M \times M}$ 
is a \textit{positive semi-definite} (PSD) matrix, denoted $Q \succeq 0$.
Hence, given a polynomial $a(\px)$, if we can find an SoS \textit{decomposition}, i.e., a PSD matrix $Q$ such that $a(\px) = \pz^\top(\px) Q \pz(\px)$, then we can ensure that $a(\px) \geq 0$ for all $\px \in \reals^D$. 


Recall, the constraints in Def. \ref{def:sbf} require bounds over a \textit{semi-algebraic set} rather than all of $\reals^D$. 
The following \edt{proposition}
allows \eqref{eq: barrier constraints} to be expressed as as non-negativity constraints.

\begin{proposition} [\!\!\cite{nie2007complexity}] \label{thm:putinar}
    Let $S = \{\px \mid \sasvec \geq 0\}$ be a semi-algebraic set. Then, a polynomial $a(\px) \in \polys$ is non-negative on $S$ if
        $a(\px) = \lambda_0(\px) + \sum_{1\leq i \leq r} \lambda_i(\px) h_i(\px)$
    for some $\lambda_0(\px), \ldots, \lambda_r(\px) \in \sospolys$ referred to as \emph{Lagrange multipliers}.
\end{proposition}
\noindent One can assert the non-negativity of $a(\px)$ on $S$ by finding a decomposition in the form described in Theorem~\ref{thm:putinar} \edt{that is SoS}. However, for our purposes it is often more useful to examine a relaxed \edt{decomposition}, as stated below.
\begin{corollary} \label{cor:relax_putinar}
    Let $S = \{\px \mid \sasvec \geq 0\}$ be a semi algebraic set. 
    Consider an ``error'' polynomial
        $u(\px) = a(\px) - \big( \lambda_0(\px) + \sum_{1\leq i \leq r} \lambda_i(\px) h_i(\px) \big)$ 
    for some $\lambda_0(\px), \ldots, \lambda_r(\px) \in \sospolys$.
    Then, a polynomial $a(\px) \in \polys$ is non-negative on $S$ if $u(\px) \geq 0$ for all $\px \in \reals^D$.
\end{corollary}
Corollary~\ref{cor:relax_putinar} allows one to use ``imperfect'' \textit{lower-bound} decompositions of $a(\px)$ in the form described in Theorem~\ref{thm:putinar} to still \edt{assert non-negativity over $S$}. 
\edt{Using Corollary \ref{cor:relax_putinar}, a certificate of non-negativity over all of $\reals^D$ can be narrowed to a bounded semi-algebraic set, allowing the constraints of Def. \ref{def:sbf} to be formulated for the polynomial $B(\px)$ in \eqref{eq:B SoS} as}:
\begin{subequations}
    \begin{align}
        &B(\px) - \sum_{i} \lambda^X_i(\px) h^X_i(\px) &&\in \sospolys, \label{eq:sos1}\\
        &B(\px)  - 1 - \sum_{i} \lambda^{X_u}_i(\px) h^{X_u}_i(\px) &&\in \sospolys, \label{eq:sos2}\\
        -&B(\px) + \eta - \sum_{i} \lambda^{X_0}_i(\px) h^{X_0}_i(\px) &&\in \sospolys, \label{eq:sos3}\\
        -&\expec{B(f(\px) + \pv)} + B(\px) + \gamma &&\notag \\ 
        &\hspace{2.8cm} - \sum_{i} \lambda^{X_s}_i(\px) h^{X_s}_i(\px) &&\in \sospolys, \label{eq:sos4}
    \end{align}
\end{subequations}
where $\lambda^S_i(\px)$ represents the corresponding Lagrange multiplier for the polynomial $h^S_i(\px)$ defining set $S$. These constraints define an SoS program 
that can be solved using SDP with an objective function of minimizing $\eta + K \gamma$ \cite{santoyo2021barrier,mazouz2022safety}. SDP is a convex program with worst-case complexity that is polynomial in the number of constraints and decision 
variables~\cite{vandenberghe1996semidefinite}.

\subsection{Understanding Lagrange Multipliers} \label{sec:lagr}
The Lagrange
multipliers in \eqref{eq:sos1}-\eqref{eq:sos4} are generally not known \textit{a priori}, and must be included as additional functionals in the optimization problem. Like $B(\px)$, we choose the maximum degree of each $\lambda_i(\px)$. This section aims to provide an intuitive and graphical perspective on the implications of the choice of $deg(\lambda)$ relative to $deg(B(\px))$.

Fig. \ref{fig:lagr_ex} illustrates an example of when Corollary \ref{cor:relax_putinar} cannot be used to certify the non-negativity over any semi-algebraic set if the degree of Lagrange multipliers is not large enough.
In this example, we aim to use $u(\px)$ to certify $a(\px) \geq 0$ (degree $6$) on $S = [0, 1]$ \edt{defined by $h(\px) = x^2 - x$}. In this case, $a(\px)$ is clearly non-negative on $[0, 1]$. On the left, $u(\px)$ is created with a quadratic Lagrange multiplier. As can be seen, in Fig. \ref{fig:lagr_ex_d2_global}, $u(\px) \not\geq 0$, and hence cannot certify that 
of $a(\px)$ is non-negative over $S$.
However, with $m_{\lambda} = 6$ (such that $m_{\lambda h} = 8$), not only is $u(\px) \geq 0$, but it achieves a much tighter lower bound of $a(\px)$ over $S$.
This suggests that when $m_{\lambda h} > m$, the existence of a valid $u(\px)$ depends on how close the minimum of $a(\px)$ on $S$ is to zero, and still benefits from raising the degree of the Lagrange multipliers. 
This example illustrates that using low-degree Lagrange multipliers can limit the expressiveness of $B(\px)$. However, increasing $m_\lambda$ \edt{increases} the number of decision variables in the SDP.

\section{Bernstein Stochastic Barrier Functions}
\label{sec:Bern SBF}

\edt{We introduce a novel \textit{alternative} method of formulating the \eqref{eq:constraint1}-\eqref{eq:constraint4}} such that the synthesis results in a \textit{linear program} (LP) instead of a SDP. This approach leverages Bernstein polynomials to bound $B(\px)$ over hyperrectangular intervals. 

\subsection{Bernstein Polynomials}
In order to constrain $B(\px)$, we must first recall the properties of Bernstein polynomials necessary to bound a multivariate polynomial on an interval.
\begin{definition}[Bernstein polynomial]
    A \emph{Bernstein Polynomial} is a maximal-degree $m$ polynomial 
    $\alpha(\px) = \sum_{\vec{l} = \vec{0}}^{\vec{m}} \beta_{l_i \ldots l_D}^{m} \phi_{\vec{l}}^m(\px),$
    where
        $\phi_l^m(\px) = \prod_{j=1}^D \binom{m}{l_j} x_j^{l_j}(1 - x_j)^{m - l_j} $
    are Bernstein basis polynomials and $\beta_{l_i \ldots l_D}^{m}$ are the corresponding coefficients.
\end{definition}

Let $a(\px)$ be an arbitrary degree-$m$ polynomial in the power basis, e.g., in the form of~\eqref{eq:B bernstein} or~\eqref{eq:B SoS}, 
with coefficients $b_{0\ldots 0}, \ldots, b_{m \ldots m}$. Additionally, let $\vec{\phi}^m$ be the organized vectorization of all degree $m$ Bernstein basis polynomials $\phi_l^m(\px)$. \edt{This section henceforth assumes maximal degree; however, Remark~\ref{rem:deg_conversion} (below) shows how one can adapt the Bernstein approach to a cumulative degree formulation.}
Then, for some raised degree $m^+ \geq m$, $a(\px) = \bvec^\top \pz(\px)$ can be equivalently represented on the Bernstein basis as $\alpha(\px) = \betavec^\top \vec{\phi}^{m^+}(\px)$.
The conversion from the power basis to the Bernstein basis coefficients is given by
\begin{align} \label{eq:tf_fwd}
    \beta_{l_1 \ldots l_D}^{m^+} = \sum_{\vec{i} = \vec{0}}^{\vec{l}} \prod_{j = 1}^D \frac{\binom{l_j}{i_j}}{\binom{m^+}{i_j}} b_{i_1 \ldots i_D}.
\end{align}
This conversion is linear, and therefore, by vectorizing the coefficients of $a(\px)$ and $\alpha(\px)$ by $\bvec$ and $\betavec$ respectively, we can represent the forward conversion in \eqref{eq:tf_fwd} using a matrix $\Phi$ such that $\betavec = \Phi_m^{m^+} \bvec$.
Work \cite{garloff1985convergent} introduces a powerful theorem that allows one to lower bound $a(\px)$ over $[0, 1]^D$.

\begin{theorem} [\!{\cite[Theorem 3]{garloff1985convergent}}] \label{thm:positivity}
    Let $\underline{a} = \inf_{\px \in [0, 1]^d} a(\px)$ and $\underline{\beta} = \min_{l} \beta_l^{m^+} $. Then, for each $m^+ \geq m$, we have
        $\underline{\beta} \leq \underline{a} \leq \underline{\beta} + \delta \frac{m^+ - 1}{(m^+)^2},$
    where $\delta \geq 0$, defined with respect to $\bvec$. Furthermore, $\underline{\beta}$ and $\underline{a}$ converge monotonically as $m^+ \rightarrow \infty$.
\end{theorem}

\subsection{Affine Transformation of a Polynomial} \label{ssec:lstf}
\edt{By constructing an affine transformation from $[0, 1]^d$ into  an arbitrary hyperrectangle $Y$, we can extend Theorem~\ref{thm:positivity} to arbitrary hyperrectangular regions.}

\begin{lemma} \label{lem:affine}
    Let $Y \subset X$ be a hyperrectangle. There exists a linear transformation $T^Y \in \reals^{M \times M}$ 
    such that $\inf_{\px \in Y} \pz(\px)^\top \bvec = \inf_{\px \in [0, 1]^d} \pz(\px)^\top T^Y \bvec$.
    
        
\end{lemma}

\begin{proof}
    Consider a change of argument to $a(\hat{\px})$ such that $\hat{x}_j = s_j x_j + t_j$. Using the binomial theorem,
        $a(\hat{\px}) = \sum_{\vec{l} = \vec{0}}^{\vec{m}} b_{l_1 \ldots l_D} \prod_{j = 1}^D \sum_{i_j = 0}^{l_j} \binom{l_j}{i_j} (s_j x_j)^{i_j}  t_j^{l_j - i_j}$.
    The power basis coefficients $\bvec'$ of the transformed polynomial can then be computed as follows
    \vspace{-2mm}
    \begin{align} \label{eq:affine}
        b'_{l_1 \ldots l_D}= \sum_{i_1 = l_1}^m \cdots \sum_{i_D = l_D}^m \prod_{j = 1}^D \binom{i_j}{l_j} (s_j)^{l_j}  t_j^{i_j - l_j} b_{i_1 \ldots i_D}.
    \end{align}
    With $Y = [\underline{y}_1, \overline{y}_1] \times \cdots \times [\underline{y}_D, \overline{y}_D] \subset X$, let $s_j = \up{y}_j - \low{y}_j$ and $t_j = \low{y}_j$.
\end{proof}

\noindent
\edt{Using \eqref{eq:affine}, constructing a transformation matrix has time complexity $O(d(m)^{2d})$ and memory complexity $O((m)^{2d})$.}


\begin{remark} \label{rem:deg_conversion}
    We restrict the Bernstein formulation to hyperrectangular regions rather than semi-algebraic sets. However, the aforementioned change-of-argument technique can extend to more complex sets via polynomial transformations.
\end{remark}

\subsection{Bernstein SBF Synthesis}

\textbf{Formulating the First Three Constraints. } \label{sec:first_three}
Using Theorem \ref{thm:positivity} and the affine transformations described above, constraints \eqref{eq:constraint1}-\eqref{eq:constraint3} can be formulated.
Suppose that sets $X$, $X_s$, $X_u$, and $X_0$ are made up of hyperrectangular regions, denoted $X_{(\cdot)}^q$, such that $X_{(\cdot)} = \bigcup_{q} X_{(\cdot)}^q$.
\begin{lemma} \label{lem:first_three}
    For the $B(\px)$ polynomial in~\eqref{eq:B bernstein}, if
    \begin{subequations} \label{eq:first_three}
    \begin{align} 
        &\Phi_m^{m^+} T^{X^q} \bvec \geq \vec{0} && \forall X^q,  \label{eq:constraint_ss}\\
        &\Phi_m^{m^+} T^{X_u^q} \bvec \geq \vec{1} && \forall X_u^q,   \label{eq:constraint_u}\\
        -&\Phi_m^{m^+} T^{X_0^q} \bvec + \eta \vec{1} \geq \vec{0} && \forall X_0^q,  \label{eq:constraint_0}
    \end{align}
    \end{subequations}
     where $T^{X_{(\cdot)}^q}$ is the affine transformation matrix for region $X_{(\cdot)}^q$ computed using~\eqref{eq:affine}, 
     then
     constraints \eqref{eq:constraint1}-\eqref{eq:constraint3} 
     hold.
\end{lemma}

\begin{proof}
    Given a partition $X_{(\cdot)}^q$, by Lemma \ref{lem:affine}: 
    \vspace{-1mm}
    \begin{equation}
        \inf_{\px \in X_{(\cdot)}^q} \pz(\px)^\top \bvec  = \inf_{\px \in [0, 1]^d} \pz(\px)^\top T^{X_{(\cdot)}^q} \bvec.
        \vspace{-1mm}
    \end{equation}
    Then, after converting to the Bernstein basis using $\Phi_m^{m^+}$, by Theorem \ref{thm:positivity}, $\inf_{\px \in X_{(\cdot)}^q} \pz(\px)^\top \bvec \geq \min_l \betavec_l^{m^+}$ where $\betavec_l^{m^+} = \Phi_m^{m^+} T^{X^q} \bvec$.
\end{proof}


\textbf{Formulating the Fourth Constraint. } \label{sec:fourth}
For noise distributions with finite moments, we can observe that, once expanded, $\expec{B(f(\px_{k}) + \pv_k)}$ is itself a composed polynomial of degree at most $p = m \sum_{d=1}^D n_d$. We can further manipulate the fourth constraint into matrix form, isolating quantities computed from the dynamics from quantities computed from the moments of the noise distribution.

\edt{Consider a ``dynamics'' matrix $F\in \mathbb{R}^{p^D \times m^D}$ where the column corresponding to the multi-index $\vec{0} \leq (i_1 \dots i_D) \leq \vec{m}$ is the coefficients of $\prod_{j=1}^D f_j(\px)^{i_j}$. The columns of $F$ correspond to monomials of a degree-$m$ polynomial composed with $f$.}
Similarly,
let $\Gamma \in \mathbb{R}^{m^D \times m^D}$ be a ``noise'' matrix with elements $\Gamma_{(l_1 \ldots l_D), (i_1 \ldots i_D)}$ characterized by the row and column corresponding to the enumeration of $(l_1 \ldots l_D)$ and $(i_1 \ldots i_D)$, respectively. The elements of $\Gamma$ are defined as
\begin{equation}
    \Gamma_{(l_1 \ldots l_D), (i_1 \ldots i_D)} = 
    \begin{cases}
        \prod_{j=1}^D \binom{l_j}{i_j} v_j^{l_j - i_j} &\text{ if } i_j \leq l_j \\
        0 &\text{ otherwise}.
    \end{cases}
\end{equation}
The element-wise expectation $\expec{\Gamma}$ of the noise matrix can be computed with the multivariate moments of $p(\pv)$.

\begin{lemma} \label{lem:fourth}
    Given the dynamics matrix $F$ and noise matrix $\Gamma$, if, 
    for every partition $X_s^q$,
    \begin{align}
        -\Phi_p^{p^+} T^{X_s^q}\big(F\, \expec{\Gamma} - \Delta_{(p^D \times m^D)} \big) \bvec + \gamma \vec{1} \geq \vec{0}, \label{eq:constraint_s}
    \end{align}
    then constraint ~\eqref{eq:constraint4} holds,
    where $\Delta_{p^D, m^D}$ is a rectangular permutation matrix that converts a degree $m$ vectorization of the coefficients into an equivalent degree $p$ vectorization, with higher order coefficients equal to zero.
\end{lemma}


\begin{proof}
    The expected barrier composed with the dynamics is
    \begin{align} \label{eq:composed}
        \expec{B(f(\px) + \pv)} &= \sum_{\vec{l} = \vec{0}}^{\vec{m}} b_{l_1 \ldots l_D} \mathbb{E}\Big[\prod_{j=1}^D (f_j(\px) + v_j)^{l_j} \Big].
    \end{align}
    \edt{Using the binomial theorem,}
    \begin{align}
        \!\!\!\mathbb{E}\Big[\! \prod_{j=1}^D (f_j(\px) + \pv)^{l_j} \! \Big] 
        \!=\! \sum_{\vec{i} = \vec{0}}^{\vec{l}} \prod_{j=1}^D f_j(\px)^{i_j}  \binom{l_j}{i_j} \expec{ v_j^{l_j - i_j}}\!. \label{eq:composed_single}
    \end{align}
    $\prod_{j=1}^D f_j(\px)^{i_j}$ is a product polynomial with coefficients equal the columns in $F$.
    \edt{By observing that $\prod_{j=1}^D \binom{l_j}{i_j} \expec{ v_j^{l_j - i_j} }$ is a scalar, the coefficients of \eqref{eq:composed_single} are linear combinations of the columns of $F$, allowing \eqref{eq:composed} to be written as $F \; \expec{\Gamma} \bvec$.}
    Finally, the coefficient vector of $\expec{B(f(\px) + \pv)} - B(\px)$ is equal to $F \; \expec{\Gamma} \bvec - \Delta_{(p^D \times m^D)} \bvec$.
    Then, similar to Lemma \ref{lem:first_three}, the affine and Bernstein transformations are applied to recover the lower bound over each safe set partition $X_s^q$.
\end{proof}
A consequence of Theorem \ref{thm:positivity} is that as $m^+ , p^+ \rightarrow \infty$, Lemmas \ref{lem:first_three} and \ref{lem:fourth} approach the biconditional implication, i.e., constraints \eqref{eq:constraint1}-\eqref{eq:constraint4} hold iff ~\eqref{eq:first_three} and ~\eqref{eq:constraint_s} are satisfied.

\begin{theorem}
    SBF synthesis for System~\eqref{eq:dynamics} using polynomial template $B(\px)$ in~\eqref{eq:B bernstein} can be formulated as a \textit{linear program}. \edt{Additionally, as $m, m^+ \rightarrow \infty$, the LP solution is the optimal SBF, making the Bernstein approach \textit{asymptotically complete}.}
\end{theorem}
\begin{proof}
    The objective function~\eqref{eq:psafe} is linear in $\eta$ and $\gamma$. $F$, $\expec{\Gamma}$ and transformation matrices $T$ can be computed beforehand, the constraint formulations in ~\eqref{eq:first_three} and ~\eqref{eq:constraint_s} are linear with respect to $\bvec$, $\eta$, and $\gamma$. \edt{For asymptotic completeness: Theorem \ref{thm:positivity} along with the Weierstrass Approximation Theorem \cite{weierstrass1885analytische} show that Bernstein polynomials are convergent polynomial relaxations and function approximations over an interval.}
\end{proof}

A LP has complexity $O(M^2 C)$ \cite{boyd2004convex} where $M$ is the number of decision variables and $C$ is the number of constraints and, is generally simpler than SDP \cite{vandenberghe1996semidefinite}.
\edt{Obtaining the optimal barrier is simply a convex linear program with $\big((m+1)^D + 2\big)$ variables and $(Q + Q_u + Q_0)(m^+ + 1)^D + Q_s(p^+ + 1)^D$ constraints where $Q$, $Q_u$, $Q_0$ and $Q_s$ are the number of hyperrectangular regions in $X$, $X_u$, $X_0$ and $X_s$ respectively.}

\begin{remark}
    For direct comparison to SoS (Sec. \ref{sec:experiments}), one can express the Bernstein formulation using cumulative-degree form by removing rows of $T^{X}$, $T^{X_u}$, $T^{X_0}$, and $F \; \expec{\Gamma} - \Delta_{(p^D \times m^D)}$ corresponding to indices $\vec{l}$ where $|\vec{l}| > m$.
\end{remark}


\section{Convergence Rates of Polynomial Bounds} \label{sec:convergence}
Sections~\ref{sec:prelim} and \ref{sec:Bern SBF} present two approaches to solving Problem~\ref{prob} using SoS and Bernstein relaxations. In this section, 
we analyze how different hyperparameters affect convergence rates in each formulation. Specifically, we introduce two methods for tightening Bernstein bounds, analogous to increasing the degree of Lagrangian multipliers in SoS relaxations as discussed in Section~\ref{sec:lagr}.

\subsection{Bernstein Relaxations}
\subsubsection{Raising the Bernstein Degree}
Recall, \eqref{eq:tf_fwd} allows one to choose a degree $m^+ \geq m$ when converting from the polynomial power basis to Bernstein basis. By Theorem \ref{thm:positivity}, the bounds on $a(\px)$ converge linearly with respect to $m^+$, i.e., $\delta \propto \frac{1}{m^+}$ where $ \delta = \underline a - \underline{\beta}$. Consequently, the number of constraints in the LP formulation increases as $O((m^+)^D)$.

\subsubsection{Subdivision} \label{sec:subd}
Alternatively, \cite{garloff1985convergent} shows that subdividing the hyperrectangular region into $\kappa^D$ regions affords Bernstein polynomial bounds that converge with respect to $\frac{1}{\kappa}$. While the number of coefficients of the Bernstein polynomial may remain fixed, this procedure increases the number of regions to check by $\kappa^D$, meaning the number of constraints in the LP increases as $O((\kappa m^+)^D)$.

\subsection{SoS Relaxations}
Work \cite{nie2007complexity} provides a convergence rate for SoS relaxations over a semi-algebraic set. Let $m_{\lambda h} = \max_{i \in \{1, \ldots, r\}}deg(\lambda_i(\px) h_i(\px))$. 
Then, increasing the degree of the Lagrangian multipliers improves the relaxation bounds at a rate $\propto \big(\log{m_{\lambda h}} - \log{c_h}\big)^{-1/c_h},$ where $c_h$ is a constant depending on the coefficients of $\sasvec$ (see \cite{nie2007complexity}). Recall that increasing the degree of the Lagrangian multipliers increases the number of decision variables. 
While problem-specific constants ($c_h$) may affect the rate of convergence, the logarithmic convergence of SoS formulations is theoretically slower than the proposed Bernstein relaxations. 
In contrast to these theoretical predictions, Section~\ref{sec:experiments} shows the convergence of Lagrangian multipliers can actually be affordable compared to techniques for Bernstein relaxations.

\ifarxiv
    
\section{Adaptive Subdivision}
Applying the aforementioned subdivision procedure to all sets $X$, $X_s$, etc. can result in exceedingly large linear programs. 
As a precursor to our practical comparison between SoS and Bernstein formulations, we present an algorithmic approach to adaptively select which hyperrectangles to subdivide, resulting in improved bounds with relatively few constraints. 

Let $A^S$ denote the constraint matrix associated with a hyperrectangular set $S$ such that constraints \eqref{eq:constraint_ss}, \eqref{eq:constraint_0}, \eqref{eq:constraint_u}, and \eqref{eq:constraint_s} can be equivalently represented as $A^S \bvec \geq \mathbf{L}_{\eta, \gamma}$ for some vector $\mathbf{L}_{\eta, \gamma} \in \reals^{|\bvec|}$ dependent on $\eta, \gamma$. Given $\bvec$, $\eta$, and $\gamma$, we define the \textit{robustness} of a given set $S$ as 
\begin{equation}
    R_{\bvec, \eta, \gamma}(S) = \min(A^S \bvec - \mathbf{L}_{\eta, \gamma}).
\end{equation}
If $\bvec$ is constrained by $S$, then $R(S) = 0$, otherwise $R(S) > 0$. 

\subsubsection{Algorithm}
For ease of presentation, we can formulate our adaptive subdivision as constrained tree-search. Let $H$ denote the set of all hyperrectangular regions $S$ such that $S \subseteq X$.
Let \textit{node} $\mathbf{n} = (Q, Q_s, Q_u, Q_0)$ be a tuple where each $Q_{(\cdot)} = \{S_1, S_2, \ldots, S_r\} \in 2^H$ is a set of hyperrectangles such that $\cup_{i=1}^r S_i = X_{(\cdot)}$.
Let the set of all nodes be denoted as $\mathbf{N}$ such that $\mathbf{n}_0 = (X, X_s, X_u, X_0) \in \mathbf{N}$.
Let $\tau : \mathbf{N} \times H \times \{1, \ldots, D\} \mapsto \mathbf{N}$ be a transition function such that $\tau(\mathbf{n}, S, I) = \mathbf{n}'$ iff for some $Q_{(\cdot)} \in \mathbf{n}$ and $Q_{(\cdot)}' \in \mathbf{n}'$, $S \in Q_{(\cdot)}$, $S \not \in Q_{(\cdot)}'$ and dividing $S$ along dimension $I$ results in two sets $S_0', S_1' \in Q_{(\cdot)}'$. 

The main idea behind our algorithm is to determine an optimal subdivision that maximizes the robustness with respect to an existing \textit{heuristic} barrier certificate $\tilde{\bvec}, \tilde{\eta}, \tilde{\gamma}$, i.e. 
\begin{equation}
    \mathbf{n}^* = \argmax_{\mathbf{n} \in \mathbf{N}} \big[ \min_{S \in \mathbf{n}} R_{\tilde{\bvec}, \tilde{\eta}, \tilde{\gamma}}(S) \big]
\end{equation}
using an abuse of notation $S\in \mathbf{n}$ meaning $S\in Q \cup Q_s \cup Q_u \cup Q_0$. The returned $\mathbf{n}^*$ is then used as a predisposition to a new LP with tighter relaxations. For practical purposes, we are interested in finding the best partitioning such that the number of constraints in the corresponding LP is below a threshold $c_{max}$. To achieve this, we present a constrained depth-first-search approach to finding $\mathbf{n}^*$ subject to $c_{max}$ in Algorithm \ref{alg:adasubd}.

\begin{algorithm}
    \caption{Adaptive Subdivision} \label{alg:adasubd}
    \begin{algorithmic}[1]
        \State $\mathbf{O} = \{\mathbf{n}_0\}$
        \While{$constraints(\mathbf{n}) < c_{max}$ where $\mathbf{n} = best(\mathbf{O})$;}
            \For{$I \in \{1, \ldots, D\}$}
                \State $\underline S \gets \argmin_{S \in \mathbf{n}}R_{\tilde{\bvec}, \tilde{\eta}, \tilde{\gamma}}(S)$
                \State $\mathbf{n}' \gets \tau(\mathbf{n}, \underline S, I)$
                \State $\mathbf{O} \gets \mathbf{O} \cup \{\mathbf{n}'\}$
            \EndFor
        \EndWhile
        \State \Return $best(\mathbf{O})$
    \end{algorithmic}
\end{algorithm}

\noindent The set $\mathbf{O}$ is an ordered set of encountered nodes such that 
\begin{equation}
    best(\mathbf{O}) = \argmax_{\mathbf{n} \in \mathbf{O}} \big[ \min_{S \in \mathbf{n}} R_{\tilde{\bvec}, \tilde{\eta}, \tilde{\gamma}}(S) \big]
\end{equation}
The current best node $\mathbf{n}$ is selected (line 2) and the minimum robustness set $\underline S \in \mathbf{n}$ is divided along each dimension (lines 3-5). If the new nodes have not been encountered before, they are added to $\mathbf{O}$ (line 6). 
The algorithm runs until $best(\mathbf{O})$ exceeds the maximum number of constraints (line 2).
\fi
\section{Experimental Evaluations} \label{sec:experiments}
This section provides a practical comparison between the SoS and Bernstein formulations for SBF synthesis. Both approaches are evaluated based on \emph{tightness} of the safety probability bound (accuracy) and \emph{computation time}. Additionally, the solutions are post-hoc checked against their constraints to ensure that the SBF does not suffer from numerical stability issues.
\edt{The experiments were performed on a 2D and 3D linear system with contractive dynamics with difference function $f(\px_k) = 0.5 \px_k + \pv_k$ where $\pv_k \sim \mathcal{N}(\mathbf{0}, 0.01 I)$. Each system was tested on a simple environment (S), and a hard environment (H) with more regions requiring a more expressive barrier. For 2D, $X = [-1, 0.5] \times [-0.5, 0.5]$, $X_0 = [-0.8, -0.6] \times [-0.2, 0.0]$, and for (H) experiments, $X_u = [-0.57, -0.53] \times [-0.17, -0.13] \cup [-0.57, -0.53] \times [-.28, 0.32]$. Boundary rectangles are placed around $X$ to ensure leaving $X$ is unsafe. The 3D experiment sets are described by $X^{(3D)}_{(\cdot)} = X^{(2D)}_{(\cdot)} \times [-1, 1]$, except $X^{(3D)}_{0} = X^{(2D)}_{0} \times [-0.1, 0.1]$.}
All experiments were performed on an Intel i7-1360P processor and were limited to 8G of RAM.
The code and details of the experiments can be found at \url{https://github.com/aria-systems-group/berry-er}.

\begin{table}[t] 
    \centering
    \caption{\edt{Experimental Results. Limited to 8G of RAM. A 2D (top) and 3D (bottom) system was verified. Two problem setups were considered: simple (S) and hard (H) where the hard problem requires a more expressive barrier. \textbf{OM} and \textbf{TO} indicate out-of-memory and time-out respectively. $M$ is the number of scalar decision variables and $C$ is the number of constraints.}}\label{tab:exp}
    \edt{
    \scalebox{0.85}{
    \begin{tabular}{|c|c|c|c|c||c|c|c|c|c|}
        \hline
                      \multicolumn{5}{|c||}{Bernstein} & \multicolumn{5}{c|}{SoS} \\ \cline{1-10} 
                      $m$ & $\kappa$ & $t$ (s) & $\delta_{\safe}$ & $M/C$ & $m$ & $m_{\lambda}$ & $t$ (s) & $\delta_{\safe}$ & $M/C$ \\ \hline
         \multicolumn{10}{|c|}{2D-S}
         \\ \hline
                      4 & 4 & 0.04 & 0.518 & 17/4k & 8 & 4 & 0.12 & 0.90 & 1k/270\\ 
                      6 & 4 & 0.09 & 0.886 & 30/7k & 12 & 4 & 0.34 & 0.952 & 3k/546\\ 
                      6 & 10 & 0.49 & 0.901 & 30/46k & 12 & 16 & 2.87 & 0.997 & 16k/1k\\ 
                      8 & 4 & 0.65 & 0.968 & 47/12k & 16 & 4 & 0.48 & 0.98 & 7k/918\\ 
                      8 & 10 & 1.91 & 0.979 & 47/78k & 16 & 20 & 1.85 & 0.998 & 33k/2k\\ 
                      10 & 4 & 0.97 & 0.988 & 68/19k & 20 & 4 & 1.44 & 0.991 & 16k/1k\\ 
                      10 & 10 & 11.2 & 0.993 & 68/116k & 20 & 24 & 2.43 & 0.998 & 65k/2k\\ 
                      14 & 4 & 68.4 & 0.997 & 112/35k & 28 & 12 & 61.0 & 0.998 & 70k/3k\\ \hline
         \multicolumn{10}{|c|}{2D-H}
         \\ \hline
                      6 & 6 & 0.64 & 0 & 30/44k & 12 & 12 & 0.9 & 0.181 & 10k/896 \\ 
                      6 & 10 & 1.90 & 0 & 30/123k & 12 & 24 & 5.62 & 0.182 & 90k/3k \\ 
                      8 & 6 & 2.76 & 0 & 47/75k & 16 & 12 & 3.00 & 0.299 & 14k/1k \\ 
                      8 & 10 & 6.01 & 0 & 47/209k & 16 & 24 & 4.17 & 0.303 & 90k/3k \\ 
                      10 & 6 & 9.92 & 0.018 & 68/114k & 20 & 12 & 4.32 & 0.299 & 25k/2k \\ 
                      10 & 10 & 42.2 & 0.021 & 68/317k & 20 & 24 & 6.87 & 0.348 & 91k/3k \\ 
                      12 & 6 & 101 & 0.176 & 93/161k & 24 & 12 & 8.28 & 0.297 & 43k/3k \\ 
                      12 & 10 & \textbf{OM} & - & - & 24 & 24 & 14.2 & 0.397 & 93k/3k \\ 
                      14 & 6 & \textbf{TO} & - & - & 28 & 12 & 23.2 & 0.300 & 70k/3k \\ 
                      14 & 10 & \textbf{OM} & - & - & 28 & 24 & 14.5 & 0.412 & 115k/3k \\ \hline
         \multicolumn{10}{c}{}
         \\ \hline
                      $m$ & $m^+$ & $t$ (s) & $\delta_{\safe}$ & $M/C$ & $m$ & $m_{\lambda}$ & $t$ (s) & $\delta_{\safe}$ & $M/C$ \\ \hline
         \multicolumn{10}{|c|}{3D-S}
         \\ \hline
                      4 & 20 & 3.29 & 0.162 & 37/39k & 8 & 12 & 2.16 & 0.976 & 79k/4k \\ 
                      5 & 20 & 8.06 & 0.211 & 58/187k & 10 & 12 & 4.94 & 0.992 & 80k/5k \\ 
                      6 & 10 & 4.34 & 0.609 & 86/64k & 12 & 4 & 6.82 & 0.93 & 32k/4k \\ 
                      6 & 20 & 17.9 & 0.732 & 86/216k & 12 & 12 & 19.0 & 0.996 & 82k/5k \\ 
                      7 & 10 & 10.3 & 0.650 & 122/79k & 14 & 4 & 38.4 & 0.961 & 66k/5k \\ 
                      7 & 20 & \textbf{OM} & - & - & 14 & 12 & 37.4 & 0.997 & 108k/5k \\ \hline
         \multicolumn{10}{|c|}{3D-H}
         \\ \hline
                      6 & 10 & 12.9 & 0 & 86/171k & 12 & 4 & 6.86 & 0 & 40k/5k \\ 
                      6 & 20 & \textbf{OM} & - & - & 12 & 12 & 42.2 & 0.180 & 118k/6k \\ 
                      7 & 10 & \textbf{OM} & - & - & 14 & 12 & 90.8 & 0.298 & 114k/7k \\ \hline
    \end{tabular}
    }
    }
    \vspace{-5mm}
\end{table}

Table \ref{tab:exp} shows the experimental results.
Each experiment varied the degree of the barrier ($m$). \edt{To ensure both barriers have comparable expressivity, the SoS degree is twice the Bernstein degree.} For Bernstein in 2D, increasing the subdivision ($\kappa$) yielded the best results; however, for 3D, only degree increase ($m^+$) yielded solutions without time-out. For SoS, the degree of the Lagrange multipliers ($m_\lambda$) was varied.
\edt{For 2D (S), the Bernstein method approaches the efficacy of SoS with many subdivisions.} 
The 2D (H) problem illustrates how the Bernstein relaxation fails to compete with SoS on the same problem and polynomial template. In 3D, the Bernstein method struggles to produce any meaningful result and runs out of memory for higher degrees, while SoS produces consistent results. 

\begin{remark}
    For the purposes of generating absolutely correct certificates of safety, one must also consider the rigorous numerical accuracy of solutions produced by either approach which is dependent on the solver, and is out of the scope of this work. \edt{In our experiments, numerical stability issues were observed with only the SoS approach when $m_\lambda \geq 28$.} Nonetheless, Bernstein relaxations typically yield very accurate results with high numerical stability \cite{ben2016linear}. 
\end{remark}

Interestingly, these trends reveal a nuanced discrepancy between the theoretical predictions and empirical results.
The struggle of the Bernstein method can mainly be attributed to a blow-up of constraints in the LP when increasing the dimension of the problem or the degree of the barrier. 
\edt{These findings show that, contrary to the preferable convergence rates, for higher dimension problems, the Bernstein method runs out of time and memory well before being able to ascertain convergence}. The Bernstein method may require better relaxation techniques, as well as problem-specific constraint reduction to match SoS, which also benefits from recent improvements of SDP tools.

\section{Conclusion and Discussion}

This work studies the difference between using Bernstein and the state-of-the-art SoS polynomial relaxations for formulating and synthesizing SBF certificates. The proposed Bernstein formulation results in a LP, which is a simpler optimization problem than SDP for the SoS formulation. However, since both Bernstein and SoS are merely \textit{relaxations} rather than exact bounds, one must tune the degree to which extra constraints and/or decision variables are added to tighten these relaxations. Beyond the simplicity of a LP, the Bernstein approach affords preferable bounds on convergence rates for the relaxations over the convergence of SoS relaxations. 
%
%

Our empirical results reveal two key insights: (i) tuning the relaxation parameters significantly impacts the safety certificate, and (ii) the Bernstein approach scales worse than SoS in terms of computation time, accuracy, and memory usage. This letter highlights a counterintuitive discrepancy between the theoretical formulation and empirical outcomes for SBF synthesis, and potentially other polynomial-relaxation procedures. Tight Bernstein relaxations require many constraints, exacerbated by the number of region of interest constraints present in SBF synthesis problems. \edt{Future work is needed to systematically reduce constraints to make the Bernstein approach comparable with SoS.} Moreover, due to the potential blow-up of constraints, we advise caution when using Bernstein polynomials as indeterminates in optimization, as opposed to applying them to bound polynomials known \textit{a priori}~as~in~\cite{garloff1993bernstein}. 

\bibliographystyle{IEEEtran}
\bibliography{bibliography}

\addtolength{\textheight}{-12cm}   

\end{document}